\documentclass[12pt]{amsart}
\usepackage[colorlinks=true,citecolor=blue,linkcolor=blue]{hyperref}
\usepackage{amssymb,amscd,amsmath,graphicx,mathtools}
\usepackage{enumerate}
\usepackage{fullpage, xcolor}
\usepackage{comment}

\usepackage[all]{xy}
\usepackage{cleveref}

\numberwithin{equation}{section}
\newtheorem{theorem}{Theorem}[section]
\newtheorem{lemma}[theorem]{Lemma}
\newtheorem{proposition}[theorem]{Proposition}

\theoremstyle{definition}
\newtheorem{definition}[theorem]{Definition}

\newtheorem{notation}[theorem]{Notation}
\newtheorem{question}[theorem]{Question}

\theoremstyle{definition}
\newtheorem{remark}[theorem]{Remark}
\newtheorem{example}[theorem]{Example}

\DeclareMathOperator{\reg}{reg}
\DeclareMathOperator{\depth}{depth}

\DeclareMathOperator{\Ass}{Ass}

\DeclareMathOperator{\Min}{Min}

\DeclareMathOperator{\gr}{gr}
\DeclareMathOperator{\Spec}{Spec}

\DeclareMathOperator{\grade}{grade}

\DeclareMathOperator{\NP}{NP}
\DeclareMathOperator{\SP}{SP}

\DeclareMathOperator{\RV}{RV}
\DeclareMathOperator{\conv}{convexhull}

\DeclarePairedDelimiter\ceil{\lceil}{\rceil}

\newcommand{\ZZ}{{\mathbb Z}}
\newcommand{\NN}{{\mathbb N}}
\newcommand{\QQ}{{\mathbb Q}}
\newcommand{\RR}{{\mathbb R}}

\newcommand{\rI}[1]{\overline{I^{{#1}}}}
\newcommand{\rsI}[1]{\overline{I^{\left({#1}\right)}}}

\def\R{{\mathcal R}}

\def\R{{\mathcal R}}

\def\mm{{\mathfrak m}}

\def\pp{{\mathfrak p}}
\def\qq{{\mathfrak q}}

\def\a{{\bf a}}
\def\b{{\bf b}}

\def\p{{\bf p}}

\newcommand{\kk}{\Bbbk}

\def\1{{\bf 1}}
\def\0{{\bf 0}}

\begin{document}
	
	\title{Rational symbolic powers of ideals}

\author{Souvik Dey}
\address[S. Dey]{Department of Mathematical Sciences
850 West Dickson Street, University of Arkansas
Fayetteville, Arkansas 72701}
\email{souvikd@uark.edu}

\author{T\`ai Huy H\`a}
\address[T. H\`a]{Tulane University, Department of Mathematics, 6823 St. Charles Ave., New Orleans, LA 70118, USA}
\email{tha@tulane.edu}

\author{Dipendranath Mahato}
\address[D. Mahato]{Tulane University, Department of Mathematics, 6823 St. Charles Ave., New Orleans, LA 70118, USA}
\email{dmahato@tulane.edu}

\author{Paolo Mantero}
\address[P. Mantero]{Department of Mathematical Sciences
850 West Dickson Street, University of Arkansas
Fayetteville, Arkansas 72701}
\email{pmantero@uark.edu}


\begin{abstract}
	We introduce and study rational symbolic powers of ideals in Noetherian rings. We give membership criteria for rational symbolic powers and discuss settings where they agree with integer symbolic powers. We investigate the binomial expansion formula for rational symbolic powers of mixed sums of ideals. Finally, we study rational symbolic powers of monomial ideals. In this case, we give a convex-geometric description of the rational symbolic powers. We also show that the filtration of rational symbolic powers of a monomial ideal is asymptotically stable and, as a consequence, deduce that the asymptotic regularity and asymptotic depth for this filtration exist.
\end{abstract}

\maketitle
	

\section{Introduction} \label{sec.intro}

\emph{Symbolic powers} of ideals have long been a fundamental object of study in commutative algebra, arising naturally in algebraic geometry through their connection with vanishing conditions along varieties. Their structure and asymptotic behavior have been central to major advances, such as the uniform containment theorems of Ein--Lazarsfeld--Smith \cite{ELS}, Hochster--Huneke \cite{HH}, and Ma--Schwede \cite{MS}, as well as the Zariski--Nagata theorem (see, for example, \cite{Dd2018}).  
In contrast, \emph{rational powers} of ideals, introduced by Swanson and Huneke \cite{sh}, are much less understood (see for instance  \cite{bisui2024rational}, or \cite{DONGRE202339} for the case of monomial ideals). Rational powers interpolate between integral closures of ordinary powers, providing a finer filtration of ideals.  

The goal of this paper is to connect these two directions by introducing and studying the notion of \emph{rational symbolic powers}. If $R$ is a Noetherian ring and $I \subseteq R$ is an ideal, we define the $u$-th rational symbolic power of $I$, for $u \in \QQ_+$, by
\[
\overline{I^{(u)}} \;=\; \bigcap_{\pp \in \Ass(I)} \left(\overline{I^u} R_\pp \cap R\right),
\]
where $\overline{I^u}$ is the $u$-th rational power of $I$ in the sense of \cite{sh}.  

Our first result, \Cref{thm.equivDef}, provides a concrete membership criterion for rational symbolic powers, in terms of usual symbolic powers and rational powers of usual symbolic powers, for ideals with locally reduced tangent cones. Specifically, we prove the following theorem.
\medskip

\noindent\textbf{Theorem \ref{thm.equivDef}.}
Let $I \subseteq R$ be an ideal. Suppose that the associated graded ring $\gr_{I_\pp}(R_\pp)$ is reduced for any $\pp \in \Ass(I)$. 
    Then, for any positive rational number $u = \frac{p}{q}$, we have
    $$\rsI{u} = \left\{ f \in R ~\big|~ f^q \in \rsI{\frac{p}{1}} \right\}=\left\{ f \in R ~\big|~ f^q \in I^{(p)} \right\}=\overline{(I^{(p)})^{1/q}}.$$

\smallskip

If, in addition, the ideals have no embedded primes, then we show in \Cref{thm.RatSymPower=SymPower} that 
rational symbolic powers agree with specific integer symbolic powers. Recall that if $u\in \QQ_+$, then $\ceil{u}$ denotes the smallest $t\in \ZZ_+$ with $t\geq u$.

\medskip

\noindent \textbf{Theorem \ref{thm.RatSymPower=SymPower}.}
	Let $R$ be a Noetherian commutative ring. Let $I \subseteq R$ an ideal and let $u = \frac{p}{q} \in \QQ_+$. Assume the following two conditions hold:
\begin{enumerate}[(i)]
\item $\Ass(I) = \Min(I)$; and 
\item The associated graded ring $\gr_{I_\pp}(R_\pp) = \bigoplus_{t \ge 0} I_\pp^t/I_\pp^{t+1}$ is reduced, for every $\pp \in \Ass(I)$.
\end{enumerate}
Then, 
$$\rsI{u} = I^{(\ceil{u})}.$$

\smallskip


Beyond these foundational results, we show that rational symbolic powers can also be realized as \emph{saturated} powers. Here, for ideals $I$ and $K$, and a positive rational number $u$, the $u$-th \emph{saturated power} of $I$ with respect to $K$ is defined as
$$\overline{I^{(u)}_K} := \rI{u} : K^\infty.$$

\medskip

\noindent\textbf{Theorem \ref{lem.symbSat_ass}.}
	Let $R$ be a Noetherian commutative ring and let $I \subseteq R$ be an ideal. Let $u$ be a positive rational number. Set
	$$K := \bigcap_{\substack{\pp \in \Ass_R^*(I) \\ \grade(\pp, R/I) \ge 1}} \pp \qquad \quad \text{ and } \qquad \quad K_u := \bigcap_{\substack{\pp \in \Ass_R(\overline{I^u}) \\ \grade(\pp,R/I) \ge 1}} \pp.$$
	Let $x \in K$ be any element that is regular on $R/I$, provided  $\{\pp \in \Ass_R^*(I) \mid \grade(\pp,R/I)\ge 1\}\neq \emptyset$, and let $x = 1$ otherwise. Then,
	$$\overline{I^{(u)}} = \overline{I^{(u)}_K} = \overline{I^{(u)}_{K_u}} = \overline{I^{(u)}_{(x)}}.$$
	In particular, an irredundant primary decomposition of $\overline{I^{(u)}}$ is obtained from an irredundant primary decomposition of $\overline{I^u}$ by intersecting only the components primary to ideals $\pp$ which are contained in some $\mathfrak q\in \Ass_R(I)$.

\medskip

\Cref{lem.symbSat_ass} has an interesting consequence, namely, for homogeneous ideals $I$ in polynomial rings over a field of characteristic 0, the rational symbolic power $\rsI{u}$ is strongly Golod for all $u \in \QQ \cap [2,\infty)$; see \Cref{golod}. Furthermore, in Theorem \ref{thm.binSat}, we generalize the binomial expansion formula for \emph{mixed} sums of ideals (extending \cite{HJKN23}) and, in Theorem \ref{thm.binRatSym}, we prove that whenever the binomial expansion formula holds for a rational power, it also holds for the corresponding rational symbolic power. 

Finally, we restrict out attention to monomial ideals in polynomial rings. In this case, we exhibit a valuative criterion and a convex-geometric description for rational symbolic powers; see \Cref{prop.RatSymMember} and \Cref{thm.ratMonIdeal}. In particular, if $I$ is a monomial ideal and $u \in \QQ_+$, then
$$\rsI{u} = \{f \in R ~\big|~ v(f) \ge u \cdot v(I) \text{ for all } v \in \RV_a(I)\},$$
where $\RV_a(I)$ is a subset of the Rees valuations of $I$. As a consequence, there exists a convex polyhedron $\Sigma(I)$ so that 
$$\rsI{u} = \langle \{x^\a ~\big|~ \a \in u \cdot \Sigma(I)\} \rangle.$$

This valuative (and convex-geometric) interpretation allows us to study asymptotic properties of the filtration of rational symbolic powers of monomial ideals. Our results establish the stability and asymptotics behaviors of this filtration for monomial ideals.  

\medskip

\noindent\textbf{Theorems \ref{thm.AsympStable} and \ref{thm.lengthCoh}.}
Let $I \subseteq R = \kk[x_1, \dots, x_n]$ be a monomial ideal. 
\begin{enumerate} 
\item The filtration $\{\overline{I^{(\frac{k}{e})}}\}_{k \in \ZZ_+}$ of rational symbolic powers of $I$ is asymptotically stable. In particular, the following limits exist:
\[
\lim_{u \to \infty} \frac{\reg \overline{I^{(u)}}}{u}
\quad \text{and} \quad
\lim_{u \to \infty} \depth (R / \overline{I^{(u)}}).
\]
\item Assume that $\lambda(H^i_\mm(R/\rsI{u})) < \infty$ for $n \gg 0$. Then, the limit
    $$\lim_{u \rightarrow \infty} \dfrac{\lambda(H^i_\mm(R/\rsI{u}))}{u^n}$$
    exists and it is a rational number.
\end{enumerate}

\smallskip

Although these statements present analogies with 
classical theorems for symbolic or rational powers, their proofs require significant new care. The rational symbolic setting introduces subtle technical challenges --- for example, handling integral closures with rational exponents, or ensuring that localization, saturation, and convex-geometric arguments still behave well under rational scaling. As a result, while our strategies are often inspired by known arguments, many proofs require substantial modifications to work in the rational context.  

\subsection*{Outline of the paper}
In Section~2, we introduce rational symbolic powers and develop their basic properties. In particular, establish the membership criterion (Theorem~\ref{thm.equivDef}), and discuss when rational symbolic powers agree with integer symbolic powers (Theorem~\ref{thm.RatSymPower=SymPower}).  

Section~3 explores the relationship between rational symbolic powers and saturated powers (Theorem \ref{lem.symbSat_ass}). Using this connection, we generalize the binomial expansion formula of H\`a--Jayanthan--Kumar--Nguyen to the rational setting (Theorem \ref{thm.binSat}). We also deduce that when the binomial expansion formula holds for rational powers, it also holds for rational symbolic powers (\Cref{thm.binRatSym}).  

In Section~4, we turn to monomial ideals. We describe rational symbolic powers of monomial ideals in terms of Rees valuations and corresponding convex polyhedron (Theorem \ref{thm.ratMonIdeal}). We also establish the asymptotic stability of the filtration of rational symbolic powers (Theorem \ref{thm.AsympStable}), which implies the existence of asymptotic regularity and depth, and prove the existence of asymptotic length of local cohomology modules of rational symbolic powers (Theorem \ref{thm.lengthCoh}).


\medskip
\noindent\textbf{Acknowledgment.} Part of this work was done while the first author visited the second and third authors at Tulane University. The authors thank Tulane University for its hospitality. Dey was partially supported by the Charles University Research Center program No.UNCE/SCI/022 and a grant GACR 23-05148S from the Czech Science Foundation. H\`a was partially supported by Simons Foundation grant \# 850912. Mantero was partially supported by Simons Foundation grant \# 962192.


\section{Rational symbolic powers} \label{sec.RatSymb}

In this section, we introduce the notion of rational symbolic powers of an ideal and explore many foundational properties.

Throughout this paper, unless otherwise stated (see Section 4) $R$ denotes a commutative Noetherian ring. Recall that the integral closure $\overline{I}$ of an ideal $I\subseteq R$ is defined as
$$
\overline{I}=\{f\in R\,\mid\,f^n + a_1f^{n-1} + \ldots + a_{n-1}f + a_n=0 \text{ for some }n\in \ZZ_+ \text{ and }a_h\in I^h \ \forall \ h\}.
$$
We shall freely use the standard fact that localization commutes with integral closures. We begin by recalling the notion of rational powers, as introduced by Swanson and Huneke \cite{sh}.

\begin{definition}[{\cite[Definition 10.5.1]{sh}}]\label{def.RationalPower}
Let $I \subseteq R$ be an ideal. For a positive rational number $u = \frac{p}{q}$, the \emph{$u$-th rational power} of $I$ is defined as
$$\rI{u} = \left\{ f \in R ~\big|~ f^q \in \rI{p}\right\}.$$
\end{definition}

The main object of study in this paper is given in the following definition.

\begin{definition}[Rational Symbolic Powers] \label{def.RationalSymbPower}
Let $I \subseteq R$ be an ideal and let $u \in \QQ$ be a positive rational number. The \emph{$u$-th rational symbolic power} of $I$ is defined as
$$\rsI{u} = \bigcap_{\pp \in \Ass(I)} \left(\rI{u} R_\pp \cap R\right).$$
\end{definition}

\begin{notation}\label{not.n/1}
 Let $I\subseteq R$ be an ideal
 \begin{itemize}
     \item  When $u \in \ZZ_+$, $\rI{u}$ coincides with the integral closure of $I^u$, by definition. However, it is not always true that the $u$-th rational symbolic power of $I$ agrees with the integral closure of $I^{(u)}$, even when $u=1$, see e.g. \Cref{Ex.u=1}. To avoid potential confusion, when $u \in \ZZ_+$, we shall write $\rsI{\frac{u}{1}}$ for the $u$-th rational symbolic power and reserve $\rsI{u}$ for the integral closure of the $u$-th symbolic power of $I$.
     \item  We shall use $e$ to denote one of the numbers specified in \Cref{rmk.RSP}(3) below.
 \end{itemize}

\end{notation}

\begin{remark} \label{rmk.RSP} 
Let $I\subseteq R$ be an ideal.
\begin{enumerate}
\item It was shown in \cite[Proposition 10.5.2]{sh} that $\rI{u}$ does not depend on the rational presentation of $u$. Thus, by definition, neither does $\rsI{u}$.
\item Since rational powers of an ideal are integrally closed, it follows easily that rational symbolic powers of an ideal are integrally closed too.
\item By \cite[Proposition 10.5.5]{sh}, there exists $e \in \ZZ_+$ (in fact, all multiples of $e$ work the same way) 
such that all rational powers of $I$ are of the form $\rI{\frac{h}{e}}$, for some $h \in \ZZ_+$. Thus, all rational symbolic powers of $I$ are of the form $\rsI{\frac{h}{e}}$, for some $h \in \ZZ_+$. 

As pointed out in \cite[Proposition 10.5.5]{sh}, $e$ can be taken to be the least common multiple of $v(I)$ when $v$ varies through all Rees valuations of $I$. (see Section \ref{sec.MonIdeal} for the definition of Rees valuations and $v(I)$.)
\item It can be seen from \Cref{def.RationalPower} that $\big\{\rI{\frac{k}{e}}\big\}_{k \in \NN}$ is a filtration of ideals. It then follows that $\big\{\rsI{\frac{k}{e}}\big\}_{k \in \NN}$ is also a filtration of ideals.
\end{enumerate}
\end{remark}

\begin{example}
    \label{ex.ratSymPower}
    Consider $I=(xy^5,x^2y^2,x^4y) \subseteq \QQ[x,y]$. Then, $\Ass(I)=\{(x),(y),(x,y)\}$ and, according to \cite[Example 4.2]{DONGRE202339}, one has
    $$\rI{\frac{4}{3}}=(x^3y^3,x^4y^2,x^2y^5).$$
   It follows easily from \Cref{def.RationalSymbPower} that $\rsI{\frac{4}{3}} = \rI{\frac{4}{3}}.$ 
\end{example}

\begin{example}
    \label{ex.ratSymPower2}
    Consider $I=(xy,yz,zx) \subseteq \QQ[x,y,z]$. Then 
   $$\rI{2}=(x^2y^2,xy^2z,x^2yz,y^2z^2,xyz^2,z^2x^2)\subsetneq \rsI{\frac{2}{1}}=(x^2y^2,y^2z^2,z^2x^2,xyz)=\overline{I^{(2)}}.$$ 
\end{example}

The following lemmas follow immediately from \Cref{def.RationalSymbPower}.

\begin{lemma}\label{lem.basic.properties}
   Let $I\subseteq R$ be an ideal.
\begin{enumerate}
\item\label{lem.localization} Let $W \subseteq R$ be any multiplicatively closed subset and let $u \in \QQ_+$. Then,
$$W^{-1}\rsI{u} = \overline{(W^{-1}I)^{(u)}}.$$

\item\label{lem.p-primaryRatPower} Let $\pp \in \Spec R$ and let $Q$ be a $\pp$-primary ideal. Then, $\overline{Q^{(u)}}$ is $\pp$-primary for all $u\in \mathbb Q_+$. 

\item\label{primdecomp} Suppose that $\Ass(I)=\Min(I)=\{\pp_1,\ldots,\pp_s\}$, and let $Q_i$ be the $\pp_i$-primary component of $I$. Then, for all  $u\in \mathbb Q_+$, we have $\overline{I^u}R_{\pp_i}\cap R=\overline{Q_i^{(u)}}$. Particularly, 
    $$\overline{I^{(u)}}=\bigcap_{i=1}^n \overline{Q_i^{(u)}}$$  
    is the irredundant primary decomposition of $\overline{I^{(u)}}$, and 
    $\Ass(\rsI{u})=\Min(\rsI{u}) = \Min(I).$
\end{enumerate}
\end{lemma}

\begin{proof} (\ref{lem.localization}) The conclusion follows from \cite[Proposition 2.4]{bisui2024rational} (see also \cite[Lemma 4.1]{lewis}), in which it was shown that rational powers commute with localization. 

(\ref{lem.p-primaryRatPower}) Since $Q^{\lceil u\rceil}\subseteq \overline{Q^u}\subseteq \sqrt Q=\pp$ for all $u\in \mathbb Q_+$, we have $\sqrt{\overline{Q^u}R_{\pp}}=\sqrt{QR_{\pp}}=\pp R_{\pp}$. Thus, $\overline{Q^u}R_{\pp}$ is $\pp R_{\pp}$-primary, and hence its contraction $\overline{Q^{(u)}}$ is $\pp$-primary.  

(\ref{primdecomp}) The statement follows from part (\ref{lem.p-primaryRatPower}).
\end{proof}

Next, for any $u\in \ZZ_+$, we compare the rational symbolic power $\overline{I^{(ut/t)}}$ (for any $t\in \ZZ_+$) with the integral closure of $I^{(u)}$, i.e. $\overline{I^{(u)}}$.

\begin{lemma}
    \label{lemma.Integer_u}
    Let $I\subseteq R$ be an ideal. Let $u \in \ZZ_+$. Then
\begin{enumerate}
\item $\overline{I^{(u)}}\subseteq \rsI{ut/t}$, for all $t\in \ZZ_+$.
\item These two ideals have the same $\pp$-primary components for all for all $\pp\in \Ass(I)$ containing $\rsI{ut/t}$.
\end{enumerate}
\end{lemma}

\begin{proof}
By \Cref{rmk.RSP}(1), it suffices to prove the inclusion when $t=1$. 
We first observe that, for any $\pp\in \Ass(I)$, we have $\overline{I^uR_\pp \cap R} = \rI{u/1}R_\pp \cap R$. The forward inclusion holds because the right-hand-side is integrally closed and it clearly contains $I^uR_\pp \cap R$. The other inclusion holds because $\rI{u/1}R_\pp \cap R=\rI{u}R_\pp \cap R = \overline{I^uR_\pp} \cap R\subseteq \overline{I^uR_\pp \cap R}$, by \cite[Prop.~10.5.2(5), Prop.~1.1.4 and Rmk~1.1.3(7)]{sh}.

(1)  Now, by the above, we have
$$\overline{I^{(u)}}= \overline{\bigcap_{\pp \in \Ass(I)} (I^u R_\pp \cap R)} \subseteq \bigcap_{\pp \in \Ass(I)} \overline{I^uR_\pp \cap R} = \bigcap_{\pp \in \Ass(I)}(\rI{u}R_\pp \cap R) = \overline{I^{(u/1)}}.$$

(2) For all $\pp\in \Ass(I)$ we have
$$
\left(\overline{I^{(u)}}\right)_\pp = \overline{I_\pp^{(u)}} = \overline{I_\pp^{u}} = \left(\overline{I^u}\right)_\pp = \left(\rI{u/1}\right)_\pp = \left(\rsI{u/1}\right)_\pp,
$$
where the second equality holds because the maximal ideal $\pp_\pp$ of $R_\pp$ lies in $\Ass(I_\pp)$, so $I_\pp^{(n)} = I_\pp^{n}$ for all $n\in \ZZ_+$, the third equality follows since rational powers commute with localization, 
and the last equality is because $\pp\in \Ass(I)$ contains $\rsI{u/1}$.
\end{proof}

\begin{remark}
    \label{rmk.Integer_u}
    Let $I\subseteq R$ be an ideal and let $u\in \ZZ_+$. 
\begin{enumerate}
   
\item In general, the inclusion $\overline{I^{(u)}}\subseteq \overline{I^{(u/1)}}$ in \Cref{lemma.Integer_u}(1) can be strict.
Indeed, in Example \ref{Ex.u=1} below we show that for any field $k$, and any polynomial ring over $k$ with $r+3$ variables, for any $r\geq 0$, there are homogeneous, \emph{almost} monomial ideals $I$, for which  $$\rsI{1}=\overline{I}\subsetneq \rsI{1/1}.$$

 \item On the other hand, we will see in \Cref{thm.equivDef} that when $I$ has locally reduced tangent cone and $u \in \ZZ_+$, the $u$-rational symbolic power of $I$ coincides with the integral closure of the usual $u$-th symbolic power of $I$, i.e., $\rsI{u} = \rsI{u/1}$.
\end{enumerate}
\end{remark}


The next example gives a family of almost monomial ideals for which $\rsI{1}\neq \rsI{1/1}$;  showing the inclusion in \Cref{lemma.Integer_u}(1) can be strict.

\begin{example}[{see \cite{LLYZ2024}}] \label{Ex.u=1}
Let $R_r:=\kk[x,y,z,w_1,\dots,w_r]$, where $k$ is an arbitrary field and $r\geq 0$ any integer. Let 
$$
W_r= \begin{cases}
(0) & \text{ if }r =0;\\
(w_1,\ldots,w_{r})^2\subseteq \kk[w_1,\ldots,w_{r}] & \text{ if }r\geq 1.
\end{cases}
$$

Consider
$$I_r:=(x^3, y^3, x^2y, x^2z-xy^2)+W_rR_r \subseteq R_r.$$
Note that $I_0=(x^3, y^3, x^2y, x^2z - xy^2)$ is precisely the ideal discussed in \cite{LLYZ2024}. 

Since $I_r=I_0\otimes_k W_r \subseteq 
\kk[x,y,z]\otimes_{\kk} \kk[w_1,\ldots,w_{r}]$, the ideal $W_rR_r$ is an integrally closed $(w_1,\ldots,w_{r})$-primary ideal, and, by \cite{LLYZ2024}, $I_0$ is $(x,y)$-primary in $\kk[x,y,z]$ with 
$$\overline{I_0}=(x^3, y^3, x^2y, x^2z, xy^2)=(x^2,y^3,xy^2)\cap (x^3,y,z),$$ we have $I_r^{(1)} = I_r$ and, by transversality of $JR_r$ with $W_rR_r$ for an ideal $J\subseteq \kk[x,y,z]$, 
$$\begin{array}{ll}
\overline{I_r} & =((x^3, y^3, x^2y, x^2z, xy^2)+W_rR_r)\\
& = [(x^2,y^3,xy^2)R_r+ W_rR_r]\cap [(x^3,y,z)R_r+W_rR_r]\\
& \subsetneq (x^2,y^3,xy^2)R_r +W_rR_r \\
& =\overline{I_r}(R_r)_{(x,y,w_1,\ldots,w_{r})}\cap R \\
&=  I_r^{(1/1)}
\end{array}$$
\end{example}


The following result gives an understanding of rational symbolic powers via integer symbolic powers.

\begin{proposition}\label[proposition]{thm.MembershipMonIdeal}
Let $I$ be an ideal in $R$. Then, for $p,q > 0$, we have
$$\rsI{p/q} = \{ f \in R ~\big|~ f^q \in \rsI{p/1}\}.$$
\end{proposition}

\begin{proof}
 By definition, for any $f \in R$, we have
\begin{align*} 
 f \in \rsI{p/q}
 & \Longleftrightarrow
 f \in \rI{p/q}R_\pp \cap R, \ \forall \ \pp \in \Ass(I) \\
 & \Longleftrightarrow
 \frac{f^q}{1} \in \overline{I^{p/1}}R_\pp, \ \forall \ \pp \in \Ass(I)\\
 & \Longleftrightarrow f^q \in \overline{I^{p/1}}R_\pp\cap R, \ \forall \ \pp \in \Ass(I) \\
 & \Longleftrightarrow f^q \in \rsI{p/1}. \qedhere
    \end{align*}
\end{proof}
  


We shall now focus on ideals $I$ whose associated rings ${\rm gr}_{I_\pp}(R_\pp)$ are reduced for all $\pp\in \Ass(I)$ and investigate how their rational symbolic powers behave. 
Our first result in this direction is a concrete description of elements belonging to rational symbolic powers. To achieve it, we will need the following remark and some auxiliary lemmas.

\begin{remark}
    \label{SH.Ex5.7}
    Let $I\subseteq R$. By \cite[Exercise 5.7]{sh}, if the associate graded ring $\gr_I(R)=\bigoplus_{n\geq 0}I^n/I^{n+1}$ is reduced then $I$ is a {\em normal} ideal, i.e. $I^n$ is integrally closed for all $n \ge 1$. 
\end{remark}

\begin{lemma}
    \label{lem.intSymb}
    Let $I \subseteq R$ be an ideal and assume that the associated graded ring $\gr_{I_\pp}(R_\pp)$ is reduced for any $\pp \in \Ass(I)$. 
    Then, for any $n \in \ZZ_+$, $I^{(n)}$ is integrally closed.
\end{lemma}

\begin{proof}
By  \Cref{SH.Ex5.7}, $I_{\p_i}^n= (IR_{\pp_i})^n$ is integrally closed for all $i=1,\ldots,s$.  Since a finite intersection of integrally closed ideals is integrally closed and the contraction of an integrally closed ideal is integrally closed, then
    $I^{(n)} = \bigcap_{i=1}^s (I^nR_{\pp_i} \cap R)$
    is integrally closed.
\end{proof}

\begin{lemma} \label{lem.IntRatSym}
Let $I \subseteq R$ be an ideal and assume that the associated graded ring $\gr_{I_\pp}(R_\pp)$ is reduced for any $\pp \in \Ass(I)$. 
Then, for any $n \in \ZZ_+$,
$$\rsI{n/1}=I^{(n)}=\overline{I^{(n)}}.$$
\end{lemma}

\begin{proof} By definition, \cite[Proposition 2.4]{bisui2024rational} and \Cref{lem.intSymb}, the $n$-th rational symbolic power of $I$ is given by $$\bigcap_{\pp \in \Ass(I)} (\overline{I^n} R_{\pp} \cap R) = \bigcap_{\pp \in \Ass(I)} (\overline{(IR_{\pp})^n} \cap R) = \bigcap_{\pp \in \Ass(I)} ((IR_{\pp_i})^n \cap R) = \bigcap_{\pp \in \Ass(I)} (I^n R_{\pp} \cap R).$$
    This is exactly the same as $I^{(n)}$. The rightmost equality follows by \Cref{lem.intSymb}.  
\end{proof}

We now collect a couple of membership criteria for rational symbolic powers when ${\rm gr}_{I_\pp}(R_\pp)$ is reduced for all $\pp\in \Ass(I)$.

\begin{theorem}
    \label{thm.equivDef}
    Let $I \subseteq R$ be an ideal. Suppose that the associated graded ring $\gr_{I_\pp}(R_\pp)$ is reduced for any $\pp \in \Ass(I)$. 
    Then, for any positive rational number $u = \frac{p}{q}$, we have
    $$\rsI{u} = \left\{ f \in R ~\big|~ f^q \in \rsI{\frac{p}{1}} \right\}=\left\{ f \in R ~\big|~ f^q \in I^{(p)} \right\}=\overline{(I^{(p)})^{1/q}}.$$
\end{theorem}

\begin{proof}
    The first equality follows from \Cref{thm.MembershipMonIdeal}. The second equality follows from \Cref{lem.IntRatSym}, and the third equality is obtained by the definition of rational symbolic powers.
\end{proof}

\Cref{Ex.u=1} illustrates that, without the assumption on the associated graded rings, the conclusion of \Cref{thm.equivDef} may not hold. 
It would be desirable to characterize classes of ideals for which the conclusion of Theorem \ref{thm.equivDef} holds.


Our next main result gives a scenario where rational symbolic powers can be realized as usual symbolic powers (with integer exponents).

\begin{theorem}\label[theorem]{thm.RatSymPower=SymPower}
Let $I \subseteq R$ be an ideal and let $u = \frac{p}{q} \in \QQ_+$. Assume that the following conditions holds:
\begin{enumerate}[(i)]
\item $\Ass(I) = \Min(I)$, and 
\item $\gr_{I_\pp}(R_\pp)$ is a reduced ring for every $\pp \in \Ass(I)$.
\end{enumerate}
Then, 
$$\rsI{u} = I^{(\ceil{u})}.$$
\end{theorem}

\begin{proof} 
By \Cref{thm.equivDef}, we have that
    $$\rsI{\frac{p}{q}} = \{f \in R ~\big|~ f^q \in I^{(p)}\}.$$
    Set $a = \ceil{u} \in \ZZ_+$. Clearly, since $aq \ge p$, if $f \in I^{(a)}$ then we obtain $f^q \in \left(I^{(a)}\right)^q \subseteq I^{(aq)} \subseteq I^{(p)}$, which yields $$\rsI{\frac{p}{q}} \supseteq I^{(a)}.$$ 

    On the other hand, since $\Ass(I) = \Min(I)$, i.e., $I$ has no embedded primes, the definition of symbolic power yields  $\Ass(I) = \Ass(I^{(m)})$ for any $m \in \ZZ_+$. Moreover, by \Cref{thm.equivDef}, for any $f \in \rsI{\frac{p}{q}}$ we have $f^q \in I^{(p)}$. Thus, for any $\pp \in \Ass(I)$, we obtain 
    $f^q/1\in \left(I^{(p)}\right)_\pp = \left(I_\pp\right)^{(p)} = \left(I_\pp\right)^p$. 

    Set $c = \text{ord}_{I_\pp}(f)$ to be the $I_\pp$-adic evaluation of $f$; that is, the largest power $d$ such that $f \in (I_\pp)^d$. It then follows that 
    $$p \le \text{ord}_{I_\pp}(f^q) = q \cdot \text{ord}_{I_\pp}(f) = qc,$$
    where the left-most equality holds since $\gr_{I_\pp}(R_\pp)$ is reduced. Therefore, $c \ge \frac{p}{q}$. Since $c \in \ZZ$, we get $c \ge \ceil{u} = a$. Hence, 
    $$f \in \left(I_\pp\right)^a = \left(I_\pp\right)^{(a)} = \left(I^{(a)}\right)_\pp.$$
    As this inclusion holds for every $\pp \in \Ass(I) = \Ass(I^{(a)})$, we conclude that $f \in I^{(a)}$. This proves  $\rsI{\frac{p}{q}} \subseteq I^{(a)}$ and, therefore, $\rsI{\frac{p}{q}} = I^{(a)}$.
\end{proof}

\begin{remark} \label[remark]{rmk.Conditions}
The hypotheses of \Cref{thm.equivDef} and \Cref{thm.RatSymPower=SymPower} hold for several classes of rings and ideals. For instance, when
\begin{itemize}
    \item $R$ is a regular domains and $I$ is a radical ideal --- in this case, $R_\pp$ is a regular local ring with maximal ideal $I_\pp = \pp R_\pp$ for $\pp \in \Ass(I)$;
    \item $R$ is complete intersection in codimension $\ge$ big-height of $I$ --- for instance, $R$ is a domain with an isolated singularity --- and $I$ is a radical ideal; 
    \item $R$ is an excellent normal domain and $I$ is a radical ideal defining surface rational singularities --- in this case, according to \cite{Cut}, $I_\pp = \pp R_\pp$ is a normal ideal; 
    \item more generally, $I$ is a radical ideal such that $\Spec R_\pp$ has a reduced tangent cone for every $\pp \in \Ass(I)$.
\end{itemize}
\end{remark}


\section{Rational symbolic powers, saturated powers and binomial expansions} \label{sec.sat}

In this section, we shall see that rational symbolic powers are saturation of rational powers. This gives rise to an interesting consequence that (sufficiently large) rational symbolic powers of homogeneous ideals in polynomial rings over a field of characteristic 0 are strongly Golod. We will also show that if the binomial expansion formula holds for a rational power of a mixed sum of ideals then it also holds for the corresponding rational symbolic powers.

We start by defining rational saturated power of an ideal with respect to another ideal. 

\begin{definition}
    \label{def.satPower}
    Let $I, K \subseteq R$ be nonzero proper ideals, and let $u \in \QQ_+$ be a positive rational number. The $u$-th \emph{saturated power} of $I$ with respect to $K$ is given by
$$\overline{I^{(u)}_K} = \rI{u} : K^\infty.$$
\end{definition}

\begin{remark} \label{rmk.satPower}
If $\rI{u} = \bigcap_{\pp \in \Ass_R(\rI{u})} Q(\pp)$ is an irredundant primary decomposition of $\rI{u}$, where $Q(\pp)$ is the $\pp$-primary component of $\rI{u}$, then
\begin{equation}
\label{eq_satpower_andprimarydecomp}
\overline{I^{(u)}_K} = \bigcap_{\pp \in \Ass_R(\rI{u}), \  K \not\subseteq \pp} Q(\pp). 
\end{equation}
\end{remark}

Note that $\bigcup_{\pp \in \Ass_R(I)} \pp$ consists of all zero-divisors of $R/I$. By prime avoidance, it is easy to see that for $y \not\in \bigcup_{\pp \in \Ass_R(I)} \pp$, we have
\begin{align} \label{eq.colon}
	\overline{I^u} : y \subseteq \ \overline{I^{(u)}} = \bigcap_{\pp \in \Ass_R(I)} (\overline{I^u}R_\pp \cap R).
\end{align}

\begin{lemma}
    \label{lem.finiteAss}
    Let $I \subseteq R$ be an ideal. Then,
$\Ass_R^*(I) := \bigcup_{u \in \QQ_+ }\Ass_R(\rI{u})$ 
is a finite set.
\end{lemma}

\begin{proof}
    It follows from \cite[Proposition 2.13]{bisui2024rational} that the Rees algebra $\bigoplus_{h \in \ZZ_+} \overline{I^{\frac{h}{e}}} t^h$ is Noetherian. The assertion now follows, for instance, from \cite[Proposition 4.2]{lewis} and Remark \ref{rmk.RSP}.
\end{proof}

Our next result realizes rational symbolic powers as rational saturated powers.

\begin{theorem}
	\label[theorem]{lem.symbSat_ass}
	Let $I \subseteq R$ be any ideal and let $u\in \QQ_+$. Set
	$$K = \bigcap_{\substack{\pp \in \Ass_R^*(I) \\ \grade(\pp, R/I) \ge 1}} \pp \qquad \text{ and }\qquad K_u = \bigcap_{\substack{\pp \in \Ass_R(\overline{I^u}) \\ \grade(\pp,R/I) \ge 1}} \pp.$$
	Let $x \in K$ be any element that is regular on $R/I$ if $\{\pp \in \Ass_R^*(I) \mid \grade(\pp,R/I)\ge 1\}\neq \emptyset$, and $x = 1$ otherwise. Then,
	$$\overline{I^{(u)}} = \overline{I^{(u)}_K} = \overline{I^{(u)}_{K_u}} = \overline{I^{(u)}_{(x)}}.$$
    In particular, an irredundant primary decomposition of $\overline{I^{(u)}}$ is obtained from an irredundant primary decomposition of $\overline{I^u}$ by intersecting only the components primary to ideals $\pp$ which are contained in some element of $\Ass_R(I)$.
\end{theorem}

\begin{proof}
If $x=1$, then $\overline{I^{(u)}_{(x)}}=\rsI{u}$. Since $K \subseteq K_u$, we have  $\overline{I^{(u)}_{K_u}} \subseteq \overline{I^{(u)}_K} \subseteq \rsI{u}= \overline{I^{(u)}_{(x)}}$. In the other case, since $(x)\subseteq K \subseteq K_u$, we have  $\overline{I^{(u)}_{K_u}} \subseteq \overline{I^{(u)}_K} \subseteq \overline{I^{(u)}_{(x)}}$. Thus, in either case one has
$$\overline{I^{(u)}_{K_u}} \subseteq \overline{I^{(u)}_K} \subseteq \overline{I^{(u)}_{(x)}},$$
and to establish the three equalities in the statement it suffices to prove the following inclusions:
	$$\overline{I^{(u)}_{(x)}} \subseteq \ \overline{I^{(u)}} \subseteq \overline{I^{(u)}_{K_u}}.$$
	
First, let $f \in \overline{I^{(u)}_{(x)}}$, then $f\in \overline{I^u} : x^h$ for some $h\in \ZZ_+$. This, together with (\ref{eq.colon}), implies that $f \in \ \overline{I^{(u)}}$ as $x^h \notin \bigcup_{\pp \in \Ass_R(I)} \pp.$ Therefore, $\overline{I^{(u)}_{(x)}} \subseteq \ \overline{I^{(u)}}$.

To establish the second inclusion, consider any element $g \in \ \overline{I^{(u)}} \setminus \overline{I^u}$, if it exists. By the definition and prime avoidance, we have $\overline{I^u} : g \not \subseteq \bigcup_{\pp \in \Ass_R(I)} \pp.$ Combining it with the inclusion $\Ass_R(\overline{I^u} :g) \subseteq \Ass_R(\overline{I^u})$ yields
$$\Ass_R(\overline{I^u} :g) \subseteq \{\pp ~\big|~ \pp \in \Ass_R(\overline{I^u}) \text{ and } \grade(\pp, R/I) \ge 1\}.$$

By Noetherianess of $R$, 
there exists $q\in \ZZ_+$ with $\left(\bigcap_{\pp \in \Ass_R(\overline{I^u}:g)} \pp\right)^q\subseteq (\overline{I^u}:g)$.  Thus,
	$$(K_u)^q =  \left(\mathop{\bigcap_{\pp \in \Ass_R(\overline{I^u}), \ \grade(\pp,R/I) \ge 1}} \pp\right)^q  \subseteq \left(\bigcap_{\pp \in \Ass_R(\overline{I^u}:g)} \pp\right)^q \subseteq \overline{I^u}: g.$$
This implies that  $g \in \overline{I^u} : K_u^{\infty}$. Hence, $\overline{I^{(u)}} \subseteq \overline{I^{(u)}_{K_u}}$.

Finally, the last assertion follows from the equality $\overline{I^{(u)}} = \overline{I^{(u)}_{K_u}}$ and Equation \eqref{eq_satpower_andprimarydecomp}.
\end{proof}

\Cref{lem.symbSat_ass} gives rise to the following interesting consequence. 
Following \cite[Section 2]{hunekegolod}, given a homogeneous ideal $I$ of $\kk[x_1,\ldots,x_d]$, where $\kk$ is a field of characteristic $0$, we say $I$ is \emph{strongly Golod} if $\partial(I)^2\subseteq I$, where $\partial(I)$ is the ideal  generated by the partial derivatives $\dfrac{\partial f}{\partial x_i}$ with $f\in I$ and $i=1,\ldots,n$. 

\begin{proposition}\label{golod}
Let $\kk$ be a field of characteristic $0$, and $I$ be a  homogeneous ideal of $\kk[x_1,\ldots,x_d]$. Then, $\overline{I^{(u)}}$ is strongly Golod for every $u\in \mathbb Q \cap [2,\infty)$.  
\end{proposition}

\begin{proof} By \Cref{lem.symbSat_ass}, setting $J$ to be the ideal generated by sufficiently large power of $(x)$, we get 
$$\overline{I^{(u)}} = \rI{u} : J=\rI{u} : J^2.$$ 
It now follows from \cite[Theorem 2.3(c)]{hunekegolod} and \cite[Corollary 4.4]{rationalgolod} that $\overline{I^{(u)}}$ is strongly Golod.  
\end{proof}

\begin{remark}\label{rmk.golod}
   If $I$ is strongly Golod, then with the same argument as in \Cref{golod}, using \cite[Corollary 4.6]{rationalgolod} and \cite[Theorem 2.3(c)]{hunekegolod}, we conclude that $\rsI{u}$ is strongly Golod for all $u \in \QQ \cap [1, \infty)$.
\end{remark}

We shall move on to investigate rational symbolic powers of mixed sums of ideals and binomial expansion formula.

\begin{notation}\label{not.A and B}
For the remaining of this section, let $A$ and $B$ be Noetherian $\kk$-algebras such that $R = A \otimes_\kk B$ is also Noetherian. 

For ideals $I\subseteq A$ and $J\subseteq B$, let $I+J$ denote the $R$-ideal $IR+JR$. We call this ideal the \emph{mixed sum} of $I$ and $J$. 
\end{notation}

\begin{definition}\label{def.bin}
Adopt \Cref{not.A and B}.    
Let $u \in \QQ_+$ be a positive rational number. We say that the \emph{binomial expansion formula} holds for $I$ and $J$ at power $u$ (cf. \cite{BH2023, bisui2024rational}) if
    $$\overline{(I+J)^u} = \sum_{\substack{0 \le w \le u \\ w \in \QQ}} \rI{w} \cdot \overline{J^{u-w}}.$$
\end{definition}

\begin{remark}\label{rmk.abusingNotation} 
By abusing notations, for an ideal $I \subseteq A$ (or $J \subseteq B$), we shall also use $I$ (respectively, $J$) to denote its extension to $R$. Particularly, summands on the right hand side of \Cref{def.bin} should be understood as $\rI{w}R \cdot \overline{J^{u-w}}R$.
\end{remark}

\begin{example}
    \label{ex.binExpRat}
    By \cite[Corollary 4.9]{bisui2024rational}, it is known that if $I$ and $J$ belong to any of the following classes of ideals, not necessarily both from the same class, then the binomial expansion formula holds for $I$ and $J$ at {\em any} rational power $u$.
    \begin{enumerate}
        \item Monomial ideals in affine semigroup rings.
        \item Sums of products of determinantal ideals of a generic matrix.
        \item Sums of products of determinantal ideals of a generic symmetric matrix.
        \item Sums of product of ideals of Pfaffians of a skew-symmetric matrix.
        \item Products of determinantal ideals of a Hankel matrix of variables.
    \end{enumerate}
\end{example}

The following result shows that if the binomial expansion formula holds for $I$ and $J$ at a rational power $u$, then the corresponding binomial expansion formula for saturated powers of $I$ and $J$ is also valid.

\begin{theorem}\label{thm.binSat}
Let $I,K \subseteq A$ and $J,L \subseteq B$ be ideals. Let $u \in \QQ_+$. If the binomial expansion formula holds for $I$ and $J$ at power $u$, then
$$\overline{(I+J)^{(u)}_{KL}} = \sum_{\substack{0 \le w \le u \\ w \in \QQ}} \overline{I^{(w)}_K} \cdot \overline{J^{(u-w)}_L}.$$
\end{theorem}

\begin{proof}
    The proof follows essentially from the same argument as that of \cite[Theorem 3.7]{HJKN23}. First, choose an $e$ such that $u = \frac{s}{e}$, for some $s \in \ZZ_+$, and rational powers and symbolic rational powers of $I, J$ and $(I+J)$ can be realized as powers with the common denominator $e$. (see \Cref{not.n/1}.) 
    
    Notice also that \cite[Lemmas 3.4--3.6]{HJKN23} apply to the filtrations $\left\{\rI{\frac{i}{e}}\right\}_{i \in \ZZ_+}$ and $\left\{\overline{J^{\frac{j}{e}}}\right\}_{j \in \ZZ_+}$. Thus, by making use of \cite[Lemma 3.6]{HJKN23} and the binomial expansion (for rational powers) of $I$ and $J$ at $u$, for any $n\in \ZZ_+$ we get the inclusion
    \begin{align*}
        \overline{(I+J)^u} :_R (KL)^n &  = \left(\overline{(I+J)^u} :_R K^n\right) :_R L^n = \left(\sum_{i=0}^s \rI{\frac{i}{e}} \cdot \overline{J^{\frac{s-i}{e}}} :_R K^n\right) :_R L^n\\
        & = \left(\sum_{i=0}^s (\rI{\frac{i}{e}} :_A K^n) \overline{J^{\frac{s-i}{e}}}\right) :_R L^n = \sum_{i=0}^s \left(\rI{\frac{i}{e}} :_A K^n\right) \left(\overline{J^{\frac{s-i}{e}}} :_B L^n\right) \\
        &\subseteq \sum_{i=0}^s \overline{I^{(\frac{i}{e})}_K} \cdot \overline{J^{(\frac{s-i}{e})}_L}.
    \end{align*}
    This inclusion is true for any $n \in \ZZ_+$. Therefore, 
    $$\overline{(I+J)^{(u)}_{KL}} \subseteq \sum_{i=0}^s \overline{I^{(\frac{i}{e})}_K} \cdot \overline{J^{(\frac{s-i}{e})}_L} = \sum_{\substack{0 \le w \le u \\ w \in \QQ}} \overline{I^{(w)}_K} \cdot \overline{J^{(u-w)}_L}.$$

    On the other hand, for every $i \ge 0$, the inclusions $I \subseteq I+J$ and $KL \subseteq K$, yield
    $$(\rI{\frac{i}{e}} :_A K^\infty)R = \rI{\frac{i}{e}} :_R K^\infty \subseteq \overline{(I+J)^{\frac{i}{e}}} :_R (KL)^\infty.$$
    Similarly, for $j \ge 0$, we have
    $$(\overline{J^{\frac{j}{e}}} :_B L^\infty)R \subseteq \overline{(I+J)^{\frac{j}{e}}} :_R (KL)^\infty.$$
    It follows that, for any $i,j \ge 0$,
    $$\overline{I^{(\frac{i}{e})}_K} \cdot \overline{J^{(\frac{j}{e})}_L} \subseteq \left(\overline{(I+J)^{\frac{i}{e}}} :_R (KL)^\infty\right) \cdot \left(\overline{(I+J)^{\frac{j}{e}}}:_R (KL)^\infty\right) \subseteq \overline{(I+J)^{\frac{i+j}{e}}} :_R (KL)^\infty.$$
    Particularly, we obtain the reverse inclusion
    $$\sum_{\substack{0 \le w \le u \\ w \in \QQ}} \overline{I^{(w)}_K} \cdot \overline{J^{(u-w)}_L} =  \sum_{i=0}^s \overline{I^{(\frac{i}{e})}_K} \cdot \overline{J^{(\frac{s-i}{e})}_L} \subseteq \overline{(I+J)^{(u)}_{KL}}.$$

\end{proof}

We shall now use this result to further establish the binomial expansion for rational symbolic powers of mixed sums of ideals.

\begin{lemma}
    \label{lem.AssQ}
    Let $\{I_i\}_{i \in \ZZ_+} \subseteq R$ and $\{J_j\}_{j \in \ZZ_+} \subseteq B$ be filtrations of ideals in $A$ and $B$. For $n \ge 1$, set
$$Q_n := \sum_{\substack{i+j = n \\ 0 \le i,j \in \ZZ}} I_iJ_j \subseteq R.$$
    Then,
        $$\Ass_R(R/Q_n) \subseteq \bigcup_{1 \le i,j \le n} \bigcup_{\substack{\pp \in \Ass_A(I_i) \\ \qq \in \Ass_B(J_{j})}} \Min_R(\pp + \qq).$$
\end{lemma}

\begin{proof}
   The standard short exact sequence
    $0 \rightarrow Q_{h-1}/Q_h \rightarrow R/Q_h \rightarrow R/Q_{h-1} \rightarrow 0$ yields $\Ass_R(R/Q_h) \subseteq \Ass_R(Q_{h-1}/Q_h) \cup \Ass_R(R/Q_{h-1})$. By induction, this implies that
    $$\Ass_R(R/Q_n) \subseteq \bigcup_{h=1}^n \Ass_R(Q_{h-1}/Q_h).$$
    By \cite[Proposition 3.3]{HNTT20}, we have
    $$Q_{h-1}/Q_h \simeq \bigoplus_{i=0}^{h-1} (I_i/I_{i+1}) \otimes_\kk (J_{h-1-i}/J_{h-i}).$$
    Thus, it follows from \cite[Theorem 2.5]{HNTT20} that
        $$\Ass_R(R/Q_n) \subseteq \bigcup_{h=1}^n \bigcup_{\substack{0\leq i \leq h-1\\ \pp \in \Ass_A(I_i/I_{i+1}) \\ \qq \in \Ass_B(J_{h-1-i}/J_{h-i})}} \Min_R(\pp + \qq).$$
    The assertion now follows because $\Ass_A(I_i/I_{i+1}) \subseteq \Ass_A(A/I_{i+1})$ and $\Ass_B(J_{h-i-1}/J_{h-i}) \subseteq \Ass_B(B/J_{h-i})$.
    
\end{proof}

\begin{remark} \label{rmk.binFor}
    Fix $u \in \QQ_+$ and nonzero proper ideals $I \subseteq A$ and $J \subseteq B$. By \Cref{rmk.RSP}(3), 
    there exists $e\in \ZZ_+$ such that rational powers of $I$ and $J$ are of the form $\overline{I^{\frac{i}{e}}}$ and $\overline{J^{\frac{j}{e}}}$, while $u = \frac{h}{e}$ for some $h \in \ZZ_+$. Consider the filtration $\{Q_{\frac{h}{e}}\}_{h \in \ZZ_+}$ of ideals, where
    $$Q_{\frac{h}{e}} := \sum_{\substack{i+j = h \\ 0 \le i,j \in \ZZ}} \overline{I^{\frac{i}{e}}} \cdot \overline{J^{\frac{j}{e}}}.$$
    This filtration arises from the families $\left\{\rI{\frac{i}{e}}\right\}_{i \in \ZZ_+}$ and $\left\{\overline{J^{\frac{j}{e}}}\right\}_{j \in \ZZ_+}$. Observe now that the binomial expansion formula holds for $I$ and $J$ at power $u$ if and only if 
    $$\overline{(I+J)^u} = Q_{\frac{h}{e}}.$$
    This is because, by \cite[Proposition 10.5.2]{sh}, $\rI{\gamma} \subseteq \rI{\beta}$ if and only if $\gamma \ge \beta$.
\end{remark}

\begin{lemma}
    \label{lem.HNTT2.9}
    Let $I \subseteq A$ and $J \subseteq B$ be nonzero proper ideals. Set
    $$K := \bigcap_{\substack{\pp \in \Ass_A^*(I) \\ \grade(\pp,A/I) \ge 1}} \pp \qquad \quad \text{ and }\qquad \quad  L := \bigcap_{\substack{\qq \in \Ass_B^*(J) \\ \grade(\qq,B/J) \ge 1}} \qq.$$
    If the binomial expansion formula holds for $I$ and $J$ at power $u$, then 
    $$\overline{(I+J)^{(u)}} = \overline{(I+J)^{(u)}_{KL}}.$$
\end{lemma}

\begin{proof}
    In light of Remark \ref{rmk.satPower} and  \Cref{lem.symbSat_ass}, it suffices to show that, for any prime ideal $P \in \Ass_R(\overline{(I+J)^u})$, we have $KL \subseteq P$ if and only if $\grade(P,R/(I+J)) \ge 1$. 

    Suppose that $KL \subseteq P$. We may assume that $K \subseteq P$. Set $\pp = P \cap A$. Then $K \subseteq \pp$, so $\pp$ contains an associated prime $\pp' \in \Ass^*_A(I)$ with $\grade(\pp',A/I) \ge 1$. That is, there exists an $A/I$-regular element $f \in \pp'$. This, together with \cite[Theorem 2.5]{HNTT20}, implies that $f$ is not contained in any of the associated primes of $R/(I+J) = A/I \otimes_\kk B/J$. Therefore, $f$ is an $R/(I+J)$-regular element, and it follows that $\grade(P, R/(I+J)) \ge 1$.

    On the other hand, suppose that $KL \not\subseteq P$. 
    Let $\pp = P \cap A$ and $\qq = P \cap B$. Let $e$ and $u = \frac{h}{e}$ be as in Remark \ref{rmk.binFor}. By Lemma \ref{lem.AssQ} and \cite[Lemma 2.1]{HNTT20}, we have that, for some $1 \le i,j \le h$, 
    $$\pp \in \Ass_A\left(\rI{\frac{i}{e}}\right), \ \qq \in \Ass_B\left(\overline{J^{\frac{j}{e}}}\right), \text{ and } P \in \Min_R({\color{blue}\pp+\qq}).$$
    Particularly, this implies that $\grade(P,R/(\pp+\qq)) = 0$.
    Using the same proof as that of \cite[Lemma 2.7]{HJKN23}, we now obtain
    \begin{align*} 
    \grade(P,R/(I+J)) & = \grade(\pp,A/I) + \grade(\qq,B/J) + \grade(P,R/(\pp+\qq)).\\
    & = \grade(\pp,A/I) + \grade(\qq,B/J).
    \end{align*}
    Thus, if $\grade(P,R/(I+J)) \ge 1$ then either $\grade(\pp,A/I) \ge 1$ or $\grade(\qq,B/J) \ge 1$. However, if $\grade(\pp,A/I) \ge 1$ (a similar argument goes for the case if $\grade(\qq,B/J) \ge 1$) then $K \subseteq \pp \subseteq P$, a contradiction to the assumption that $KL \not\subseteq P$. Hence, we must have $\grade(P,R/(I+J)) = 0$. The assertion is proved.
\end{proof}

\begin{theorem} \label{thm.binRatSym}
Let $I \subseteq A$ and $J \subseteq B$ be nonzero proper ideals. If the binomial expansion formula holds for $I$ and $J$ at power $u \in \QQ_+$, i.e., if 
$$\overline{(I+J)^u} = \sum_{\substack{0 \le w \le u \\ w \in \QQ}} \overline{I^w} \cdot \overline{I^{u-w}},$$
then the binomial expansion formula holds for $I$ and $J$ at symbolic power $u \in \QQ_+$, i.e.,
$$\overline{(I+J)^{(u)}} = \sum_{\substack{0 \le w \le u \\ w \in \QQ}} \overline{I^{(w)}} \cdot \overline{I^{(u-w)}}.$$
\end{theorem}

\begin{proof}
    The assertion follows by combining Theorem \ref{thm.binSat}, \Cref{lem.symbSat_ass} and Lemma \ref{lem.HNTT2.9}.
\end{proof}


\begin{example}
    \label{ex.binExpRatSym}
    By \Cref{ex.binExpRat} and \Cref{thm.binRatSym}, it follows that if $I$ and $J$ belong to any of the following classes of ideals (not necessarily the same class), then the binomial expansion for rational symbolic powers of $I$ and $J$ holds.
    \begin{enumerate}
        \item Monomial ideals in affine semigroup rings.
        \item Sums of products of determinantal ideals of a generic matrix.
        \item Sums of products of determinantal ideals of a generic symmetric matrix.
        \item Sums of product of ideals of Pfaffians of a skew-symmetric matrix.
        \item Products of determinantal ideals of a Hankel matrix of variables.    
    \end{enumerate}
\end{example}

\begin{question}
    Does the binomial expansion formula for rational symbolic powers hold for any ideals $I \subseteq A$ and $J \subseteq B$?
\end{question}


\section{Rational symbolic powers of monomial ideals} \label{sec.MonIdeal}

In this section, we restrict our attention to monomial ideals. We shall examine convex bodies that describe rational symbolic powers, and study the asymptotic behaviors and invariants of the filtration of rational symbolic powers of a monomial ideal. Throughout this last section, $R = \kk[x_1, \dots, x_n]$ now denotes a polynomial ring over a field $\kk$.

We start with a simple observation that rational symbolic powers of a monomial ideal are also monomial ideals. 

\begin{lemma}
    \label{lem.ratIsMon}
    Let $I$ be a monomial ideal in $R$ and let $u \in \QQ_+$. Then, $\rsI{u}$ is a monomial ideal.
\end{lemma}

\begin{proof}
    By definition $\rsI{u} = \bigcap_{\pp \in \Ass(I)} (\rI{u} R_\pp \cap R)$, so it suffices to show that for any $\pp \in \Ass(I)$, $\rI{u} R_\pp \cap R$ is a monomial ideal in $R$.
    
Since $\rI{u}$ is a monomial ideal (see \cite{sh}), we have an irredundant primary decomposition of $\rI{u}$ consisting of monomial ideals. Then $\rI{u} R_\pp \cap R$ is the intersection of all such primary components whose radical is contained in $\pp$. Being the intersection of monomial ideals, $\rI{u} R_\pp \cap R$ is a monomial ideal as well.
\end{proof}

%

We shall recall the notion of \emph{Rees valuations} of an ideal following \cite{sh}.

\begin{definition}[{\cite{sh}}]\label[definition]{def.ReesVal}
Let $I = (f_1, \dots, f_d)$ be an ideal in $R$. For $i = 1, \dots, d$, set $R_i = R\left[\frac{f_1}{f_i}, \dots, \frac{f_d}{f_i}\right]$ and $\overline{R_i}$ its integral closure in $\text{QF}(R_i) = \text{QF}(R)$. Let $\RV(I)$ be the set of distinct discrete valuation rings arising as $V = \left(\overline{R_i}\right)_\pp$, where $\pp$ varies over the minimal primes of $I\overline{R_i} = f_i\overline{R_i}$. Elements in $\RV(I)$ are called the \emph{Rees valuation rings} of $I$. Discrete valuations of these Rees valuation rings are called \emph{Rees valuations} of $I$. 
\end{definition}

By abusing notation, we will often identify each $V \in \RV(I)$ with any valuation $v$ which is a positive integer multiple of the order valuation of $V$, and think of $\RV(I)$ also as the set of Rees valuations of $I$. For a valuation $v$ of $R$, define the \emph{center} of $v$ to be
$$c(v) = \{f \in R ~\big|~ v(f) > 0\},$$
and, for an ideal $I \subseteq R$, set 
$$v(I) = \min\{v(f) ~\big|~ f \in I \setminus \{0\}\}.$$

\begin{remark} \label[remark]{rmk.ReesVal}
Let $I \subseteq R$ be a monomial ideal. By \cite[Theorem 10.3.5]{sh}, the Rees valuations of $I$ are \emph{monomial valuations} obtained from the bounding hyperplanes of the Newton polyhedron of $I$. Particularly, for each $v \in \RV(I)$, there exists $\lambda = (\lambda_1, \dots, \lambda_n) \in \ZZ^n_{\ge 0}$ such that
$$v(x^\a) = \langle \lambda, \a\rangle =\lambda_1a_1 + \ldots + \lambda_n a_n.$$
Furthermore, the center of $v$ is a monomial prime ideal in $R$ (consisting of $x_i$ with $v(x_i) > 0$).
\end{remark}

It is known that the integral closures of an ideal $I$ and its powers are determined by Rees valuations, e.g. \cite[Theorem 6.8.3 and Theorem 10.2.2]{sh}. More generally, in \cite[Proposition 10.5.2]{sh} (see also \cite[Proposition 2.8]{bisui2024rational}), it was shown that for any $u \in \QQ_+$, 
\begin{align} 
\rI{u} = \{f \in R ~\big|~ v(f) \ge u\cdot v(I), \ \forall \ v \in \RV(I)\}, \label{eq.RatPowerRees}
\end{align}
and no valuation $v \in \RV(I)$ can be omitted without breaking some of these equalities.

\begin{definition}\label[definition]{def.ReesLocal}
Let $I \subseteq R$ be a monomial ideal.
\begin{enumerate}
    \item For $\pp \in \Ass(I)$, set $\RV_\pp(I) = \{ v \in \RV(I) ~\big|~ c(v) = \pp\}.$
    \item Set $\RV_a(I) = \bigcup_{\pp \in \Ass(I)} \RV_\pp(I).$
\end{enumerate}
\end{definition}

The following lemma is a valuative criterion for rational symbolic powers.

\begin{proposition}\label[proposition]{prop.RatSymMember}
Let $I \subseteq R$ be a monomial ideal. For $u \in \QQ_+$ and any monomial $x^\a \in R$, we have
$$x^\a \in \rsI{u} \Longleftrightarrow v(x^\a) \ge u \cdot v(I) \text{ for all } v \in \RV_a(I).$$
Equivalently,
$$\rsI{u} = \{f \in R ~\big|~ v(f) \ge u \cdot v(I) \text{ for all } v \in \RV_a(I)\}.$$
\end{proposition}

\begin{proof}
Fix $\pp \in \Ass(I)$. Since $IR_\pp$ is $\pp R_\pp$-primary, by \cite[Proposition 10.4.1 and Theorem 10.4.2]{sh}, we have
    $$\RV(IR_\pp) = \{v \in \RV(I) ~\big|~ c(v) = \pp\} = \RV_\pp(I).$$
Moreover, for $v \in \RV(IR_\pp)$, we then have $v(IR_\pp) = v(I)$.
    It now follows from \cite[Proposition 2.8]{bisui2024rational} (see also (\ref{eq.RatPowerRees})) that
    $$f \in \rI{u}R_\pp = \overline{(IR_\pp)^u}\Longleftrightarrow v(f) \ge u \cdot v(I) \text{ for all } v \in \RV_\pp(I).$$
    
    The first assertion in the statement is now obtained by taking the intersection of these conditions for all $\pp \in \Ass(I)$. The second assertion follows because $\rsI{u}$ is a monomial ideal (see Lemma \ref{lem.ratIsMon}) and Rees valuations of $I$ are monomial valuations. (\cite[Proposition[10.3.4]{sh}.)
\end{proof}

Recall that the \emph{Newton polyhedron} of a monomial ideal $I \subseteq R$ is 
			$$\NP(I) = \conv \langle \{ \a \in \ZZ_{\ge 0}^n ~\big|~ x^\a \in I\} \rangle \subseteq \RR_{\ge 0}^n. $$
One defines the following \emph{symbolic region} associated to a monomial ideal $I$:
$$\Sigma(I) = \bigcap_{v \in \RV_a(I)} \{ \a \in \RR^n_{\ge 0} ~\big|~ v(x^\a) \ge v(I)\}.$$
Note that, when $I$ is a squarefree monomial ideal, this symbolic region is the same as the \emph{symbolic polyhedron} defined in \cite{CEHH}. As observed in Remark \ref{rmk.ReesVal}, each inequality $v(x^\a) \ge v(I)$ has the form 
$$\lambda_1 a_1 + \dots + \lambda_n a_n \ge v(I),$$
with $\lambda_i \in \ZZ_{\ge 0}$ and $v(I) \in \ZZ_{\ge 0}$. Thus, $\Sigma(I)$ is a rational polyhedron with integral facet normals. Our next result shows that $\Sigma(I)$ gives a convex-geometric understanding of rational symbolic powers of $I$. When $I$ is a squarefree monomial ideal, this result was given in \cite{DFMS2019} in terms of the symbolic polyhedron $\SP(I)$ of $I$.

\begin{theorem}\label[theorem]{thm.ratMonIdeal}
Let $I \subseteq R$ be a monomial ideal. Then, for any $u \in \QQ_+$, we have
$$\rsI{u} = \left(\{x^\a ~\big|~ \a \in u \cdot \Sigma(I)\} \right).$$
\end{theorem}

\begin{proof}
    The conclusion follows directly from  \Cref{prop.RatSymMember}.
\end{proof}

 The remaining of this section focuses on asymptotic behaviors and invariants of the filtration of rational symbolic powers of a monomial ideal. Particularly, for a monomial ideal $I \subseteq R$, we consider the following limits
 $$\lim_{u \rightarrow \infty} \dfrac{\reg \rsI{u}}{u}, \qquad \ \lim_{u \rightarrow \infty} \depth R/\rsI{u}, \qquad \text{ and } \qquad \lim_{u \rightarrow \infty} \dfrac{\lambda(H^i_\mm(R/\rsI{u}))}{u^n}.$$
 We shall make use of the splitting techniques introduced by Monta\~no and N\'u\~nez-Betancourt \cite{splitting} and Lewis \cite{lewis}, together with the following lemma.

 \begin{lemma}
     \label{lem.NoetherianRees}
Let $I\subseteq R$ be any ideal whose symbolic Rees algebra $\R_s(I) = \bigoplus_{h \ge 0}I^{(h)}t^h$ is Noetherian. Then, $\mathcal{S} := \bigoplus_{h \ge 0}\rsI{\frac{h}{1}}t^h$ is a finitely generated module over $\R_s(I)$.
 \end{lemma}

 \begin{proof}
     By \cite[Theorem 5.3.4]{sh}, there exists $h_0 \in \ZZ_+$, such that for every $h \ge h_0$, we have
     $$\rI{\frac{h}{1}} = \overline{I^h} = I^{h-h_0}\cdot \overline{I^{h_0}} \subseteq I^{h-h_0}.$$
     This, together with the fact that rational powers commute with localization, implies that 
     $$
     \rsI{\frac{h}{1}} = \bigcap_{\pp \in \Ass(I)} (\rI{\frac{h}{1}}R_\pp \cap R) \subseteq \bigcap_{\pp \in \Ass(I)} (I^{h-h_0}R_\pp) \cap R) = I^{(h-h_0)}.$$
   It follows that $\mathcal{S} \subset \mathcal{T}$, where $\mathcal{T} = \bigoplus_{\ell=0}^{h_0} \R_s(I)\cdot t^\ell$. It is easy to verify that $\mathcal{S}$ is in fact an $\R_s(I)$-submodule of $\mathcal{T}$. Since $\R_s(I)$ is Noetherian and $\mathcal{T}$ is a finitely generated $\R_s(I)$-module, we conclude that $\mathcal{S}$ is also a finitely generated $\R_s(I)$-module, as desired.
 \end{proof}

 For $m\in \ZZ_+$, define $R^{1/m}=\kk[x_{1}^{1/m},\ldots,x_{n}^{1/m}]$ with a set of new variables $x_{1}^{1/m},\ldots,x_{n}^{1/m}$ over $\kk$. Then, $R^{1/m}$ is isomorphic to $R$ under the ring map $x_{i}\mapsto x_{i}^{1/m}$, for $i=1,\ldots,n$. Furthermore, for a monomial ideal $I\subseteq R$  with minimal monomial generating set $G(I)$, consider the ideal $I^{1/m}:=(f^{1/m} ~\big|~ f\in G(I))$ in $R^{1/m}$. Observe that $R\subseteq R^{1/m}$ and there is a natural inclusion $\iota:R\hookrightarrow{}R^{1/m}$ by mapping $x^\a \mapsto \left(x^{m\cdot \a}\right)^{1/m}$. Now we define a splitting from $R^{1/m}$ to $R$ as follows:
 
\begin{definition} \label{def.Phi}
   Let $\Phi_{m}^{R}:R^{1/m}\to R$ be the ring homomorphism which is induced by the following map  
   \begin{equation}  \Phi_{m}^{R}\left(\left(x^\a\right)^{1/m}\right)=
       \begin{cases}
           x^{\frac{1}{m}\cdot \a},  &\text{if }\a \equiv \mathbf{0} (\text{mod  }m)\\
           0, &\text{otherwise.}
       \end{cases}
   \end{equation}
\end{definition}


\begin{definition}[{\cite[Definition 2.1]{lewis}}]
     \label{def.AsymStable}
     A filtration $\{I_h\}_{h \in \ZZ_+}$ of monomial ideals in $R$ is called \emph{asymptotically stable} if the following conditions hold:
     \begin{enumerate}
         \item The Rees algebra $\bigoplus_{h \in \ZZ_+} I_ht^h$ is Noetherian;
         \item For an unbounded sequence $\{m_\ell \}_{\ell \in \ZZ_+}$, $\iota$ and $\Phi^R_{m_\ell}$ induce a split injection $\vartheta: R/I_{kh1} \rightarrow R^{1/m_\ell}/(I_{hm_\ell + j})^{1/m_\ell}$ for all $h,\ell \in \ZZ_+$ and $1 \le j < m_\ell$.
     \end{enumerate}
 \end{definition}
 
 \begin{theorem}
     \label{thm.AsympStable}
     Let $I \subseteq R$ be a monomial ideal. Then, the filtration $\{\overline{I^{(\frac{h}{e})}}\}_{h \in \ZZ_+}$ is asymptotically stable. As a consequence, the following limits exist:
       $$\lim\limits_{h \rightarrow \infty} \frac{\reg \rsI{\frac{h}{e}}}{h} \qquad \text{ and }\qquad 
       \lim\limits_{u \rightarrow \infty} \depth R\Big/\rsI{\frac{h}{e}}.$$
 \end{theorem}

\begin{proof}
We shall first show that the Rees algebra $\R = \bigoplus_{h \in \ZZ_+} \overline{I^{(\frac{h}{e})}} t^h$ is Noetherian. 
It is enough to exhibit a Veronese sub-algebra of $\R$ that is finitely generated. To this end, we consider the $e^{th}$-Veronese subalgebra of $\R$, namely,
$$\R^{\langle e\rangle}=\bigoplus_{s\ge0}\overline{I^{(\frac{se}{e})}}t^{se}=\bigoplus_{s\ge0}\overline{I^{\left(\frac{s}{1}\right)}}t^{se}=\bigoplus_{s\ge0}\overline{I^{\left(\frac{s}{1}\right)}}u^{s},$$ 
where $u:=t^{e}$. By \cite[Theorem 3.2]{herzog2006symbolicpowersmonomialideals}, it is known that the usual symbolic Rees algebra $\R_s(I)$ of $I$ is Noetherian, so it follows from Lemma \ref{lem.NoetherianRees} that $\bigoplus_{s \ge 0}\rsI{\frac{s}{1}} t^s$ is Noetherian. This, together with \cite[Lemma 1.1(a)]{Ulrich.Hong}, implies that $\R^{\langle e\rangle}$ is Noetherian.

We shall now prove that, for an unbounded sequence $\{m_{\ell}\}_{\ell\in\ZZ_+}\subseteq\ZZ_+$, $\iota$ and $\displaystyle\Phi_{m_\ell}^{R}$ induce a split injection $$\displaystyle \vartheta: \left.R\middle/{\overline{I^{(\frac{h+1}{e})}}}\right. \longrightarrow \left.R^{1/m_\ell}\middle/\left({\overline{I^{(\frac{hm_{\ell}+j}{e})}}}\right)^{1/m_\ell}\right.$$ for all $h,e\in \ZZ_+$ and $1\le j\le m_{\ell}.$
Since $\displaystyle \Phi_{m_\ell}^{R}$ is a splitting of $\iota$, it is enough to verify the following:
    \begin{enumerate}
        \item[(a)] $\displaystyle \iota\left({\overline{I^{(\frac{h+1}{e})}}}\right)\subseteq \left({\overline{I^{(\frac{hm_{\ell}+j}{e})}}}\right)^{1/m_\ell}, \text{ for }1\le j\le m_{\ell}$,
        \item[(b)] $\displaystyle \Phi_{m_\ell}^{R}\left(\left({\overline{I^{(\frac{hm_{\ell}+j}{e})}}}\right)^{1/m_\ell}\right)={\overline{I^{(\frac{h+1}{e})}}}, \text{ for }1\le j\le m_{\ell}$.
    \end{enumerate}

Indeed, to see (a), observe that, for $1\le j\le m_{\ell}$,
        \begin{equation*} 
            \begin{split}
                \displaystyle \iota\left({\overline{I^{(\frac{h+1}{e})}}}\right)&\overset{\text{def}}{\subseteq}\left({\overline{I^{(\frac{m_{\ell}(h+1)}{e})}}}\right)^{1/m_\ell}
               = \left({\overline{I^{(\frac{m_{\ell}h+m_{\ell})}{e})}}}\right)^{1/m_\ell}
               \overset{}{\subseteq}\left({\overline{I^{(\frac{hm_{\ell}+j}{e})}}}\right)^{1/m_\ell}.
            \end{split}
        \end{equation*}

  To prove (b), applying $\Phi_{m_{\ell}}^{R}$ on both side of the inclusion in (a), we get 
$$\Phi_{m_{\ell}}^{R}\left(\iota\left({\overline{I^{(\frac{h+1}{e})}}}\right)\right)\subseteq\Phi_{m_{\ell}}^{R}\left(\left({\overline{I^{(\frac{hm_{\ell}+j}{e})}}}\right)^{1/m_\ell}\right).$$
That is,
 $${\overline{I^{(\frac{h+1}{e})}}} \subseteq\Phi_{m_{\ell}}^{R}\left(\left({\overline{I^{(\frac{hm_{\ell}+j}{e})}}}\right)^{1/m_\ell}\right).$$

It remains to show that $\displaystyle \Phi_{m_{\ell}}^{R}\left(\left({\overline{I^{(\frac{hm_{\ell}+j}{e})}}}\right)^{1/m_\ell}\right)\subseteq {\overline{I^{(\frac{h+1}{e})}}}$ for any $h \in \ZZ_+$ and $1 \le j \le m_\ell$.
Indeed, consider any monomial $\left(x^\a\right)^{1/m_\ell} \in \left({\overline{I^{(\frac{km_{\ell}+j}{e})}}}\right)^{1/m_\ell}$ such that $\Phi^R_{m_\ell}\left(\left(x^\a\right)^{1/m_\ell}\right) \not= 0$. Then, $x^\a \in \overline{I^{(\frac{km_{\ell}+j}{e})}}$ and $\a \equiv \mathbf{0} (\text{mod } m_\ell)$. 

By \Cref{prop.RatSymMember} (see also Theorem \ref{thm.ratMonIdeal}), $x^\a \in \rsI{\frac{hm_\ell+j}{e}}$ is determined by the condition that $v(x^\a) \ge \frac{hm_\ell + j}{e} \cdot v(I)$ for any valuation $v \in \RV_a(I)$. (equivalently, any monomial valuation $v$ corresponding to the defining half-space of $\Sigma(I)$.) That is, $v(x^{e \cdot \a}) \ge (hm_\ell+j) \cdot v(I)$ for any $v \in \RV_a(I)$.

Set $\b = \frac{\a}{m_\ell} = \Phi^R_{m_\ell}\left(\left(x^\a\right)^{1/m_\ell}\right)$. Then, for any monomial valuation $v \in \RV_a(I)$, we get
$$v(x^{e \cdot \b}) = \frac{1}{m_\ell} v(x^{e \cdot \a}) \ge \left(h+ \frac{j}{m_\ell}\right) \cdot v(I) > h \cdot v(I).$$
On the other hand, as shown in \cite[Proposition 10.5.5]{sh} (see also \Cref{rmk.RSP}(3)), $e$ is chosen to be (any multiple of) the least common multiple of $\{v(I) ~\big|~ v \in \RV(I)\}$. Particularly, it follows that $e$ is divisible by $v(I)$ for any $v \in \RV_a(I)$. Since $v(x^{e \cdot \b}) = e \cdot v(x^\b)$, it follows that
$$v(x^{e \cdot \b}) \ge (h+1) \cdot v(I).$$
This implies that $v(x^{\b}) \ge \frac{h+1}{e} \cdot v(I)$ for any $v \in \RV_a(I)$. Equivalently, by \Cref{prop.RatSymMember}, we get $x^\b \in \rsI{\frac{h+1}{e}}$ as desired. 

This concludes the proof that the filtration of rational symbolic powers $\{\rsI{\frac{h}{e}}\}_{h \in \ZZ_+}$ is asymptotically stable. The second statement of the theorem follows from \cite[Theorem 2.1]{lewis}.
\end{proof}

\begin{remark} \label{rmk.asymptoticGen}
As observed in \cite{ha2025asymptoticregularitygradedfamilies}, for a homogeneous ideal $I \subseteq R$ in general, it is not expected that the asymptotic regularity $\lim\limits_{h \rightarrow \infty} {\reg I^{(h)}}/{h}$ exists. Thus, in general, the filtration $\big\{I^{(\frac{h}{e})}\big\}_{h \in \ZZ_+}$ is not expected to be asymptotically stable.
\end{remark}

We conclude the paper with the following result, where $\lambda(\bullet)$ denotes the length function.

\begin{theorem}
    \label{thm.lengthCoh}
    Let $I \subseteq R$ be a monomial ideal. Assume that $\lambda(H^i_\mm(R/\rsI{u})) < \infty$ for $n \gg 0$. Then, the limit
    $$\lim_{u \rightarrow \infty} \dfrac{\lambda(H^i_\mm(R/\rsI{u}))}{u^n}$$
    exists and is a rational number.
\end{theorem}

\begin{proof}
    The proof follows in exactly the same argument as that of \cite[Theorem 4.3]{lewis}, referring to a rational symbolic version of \cite[Lemma 4.2]{lewis}, whose proof is obtained by replacing the Newton polyhedron $\NP(I)$ of a monomial ideal $I$ by its symbolic region $\Sigma(I)$. 
\end{proof}

\begin{question} What are other classes of ideals for which the conclusions of Theorems \ref{thm.AsympStable} and \ref{thm.lengthCoh} hold?
\end{question}



\bibliographystyle{plain}
\bibliography{main}

\end{document}